\documentclass[12pt]{amsart}
\usepackage[latin1]{inputenc}
\usepackage{color}
\usepackage{enumerate}
\usepackage{amssymb}
\usepackage[all]{xy}
\usepackage{hyperref}

\newtheorem{thm}{Theorem}[section]

\newtheorem{lem}[thm]{Lemma}
\newtheorem{prop}[thm]{Proposition}
\theoremstyle{definition}
\newtheorem{defn}[thm]{Definition}
\theoremstyle{remark}

\numberwithin{equation}{section}
\newcommand{\norm}[1]{\left\Vert#1\right\Vert}
\newcommand{\abs}[1]{\left\vert#1\right\vert}
\newcommand{\set}[1]{\left\{#1\right\}}

\newcommand{\To}{\longrightarrow}

\usepackage{color}
\usepackage{enumerate}
\begin{document}
\setcounter{tocdepth}{1}


\title[The free Banach lattice generated by a lattice]{The free Banach lattice generated by a lattice}
\author[A.\ Avil\'es]{Antonio Avil\'es}
\address{Universidad de Murcia, Departamento de Matem\'{a}ticas, Campus de Espinardo 30100 Murcia, Spain.}
\email{avileslo@um.es}

\author[J.D. Rodr\'iguez Abell\'an]{Jos\'e David Rodr\'iguez Abell\'an}
\address{Universidad de Murcia, Departamento de Matem\'{a}ticas, Campus de Espinardo 30100 Murcia, Spain.}
\email{josedavid.rodriguez@um.es}

\thanks{Authors supported by projects MTM2014-54182-P and MTM2017-86182-P (Government of Spain, AEI/FEDER, EU) and by Fundaci\'on S\'eneca: project 19275/PI/14 for the first author and FPI contract for the second author.}



\begin{abstract}
We introduce the free Banach lattice generated by a lattice $\mathbb{L}$. We give an explicit description of it and we study some of its properties for the case when $\mathbb{L}$ is a linear order, like the countable chain condition. 
\end{abstract}

\maketitle

\setlength{\parskip}{4mm}

\emph{Remark: This version differs from the published version in that it contains a new proof of Lemma 5.3. The previous proof worked for the particular case in which the lemma was applied, but it was wrong for the lemma as stated in general.}

\section{Introduction}

The purpose of this paper is to introduce the free Banach lattice generated by a lattice and investigate some of its properties. The free Banach lattice generated by a set $A$ with no extra structure, which is denoted by $FBL(A)$, has been recently introduced and analyzed by B. de Pagter and A.W. Wickstead in \cite{dPW15}, while the free Banach lattice generated by a Banach space $E$ has been studied by A. Avil\'{e}s, J. Rodr\'{i}guez and P. Tradacete in \cite{ART18}.

If $A$ is a set with no extra structure, $FBL(A)$ is a Banach lattice together with a bounded map $u : A \longrightarrow FBL(A)$ having the following universal property: for every Banach lattice $Y$ and every bounded map $v : A \longrightarrow Y$ there is a unique Banach lattice homomorphism $S : FBL(A) \longrightarrow Y$ such that $S \circ u = v$ and $\norm{S} = \sup \set{\norm{v(a)} : a \in A}$. The same idea is applied by A. Avil\'{e}s, J. Rodr\'{i}guez and P. Tradacete to define the concept of the free Banach lattice generated by a Banach space $E$, $FBL[E]$. This is a Banach lattice together with a bounded operator $u:E\To FBL[E]$  such that for every Banach lattice $Y$ and every bounded operator $T : E \longrightarrow Y$ there is a unique Banach lattice homomorphism $S : FBL[E] \longrightarrow Y$ such that $S \circ u = T$ and $\norm{S} = \norm{T}$.

We can consider a similar idea using lattices instead of Banach spaces. Remember that a lattice is a set $\mathbb{L}$ together with two operations $\wedge$ and $\vee$ that are the infimum and supremum of some partial order relation on $\mathbb{L}$, and a lattice homomorphism is a function between lattices that commutes with the two operations. 

\begin{defn} \label{FBLGBL} Given a lattice $\mathbb{L}$, the \textit{free Banach lattice generated by $\mathbb{L}$} is a Banach lattice $F$ together with a lattice homomorphism $\phi: \mathbb{L} \longrightarrow F$ such that for every Banach lattice $X$ and every bounded lattice homomorphism $T: \mathbb{L} \longrightarrow X$, there exists a unique Banach lattice homomorphism $\hat{T}: F \longrightarrow X$ such that $|| \hat{T} || = || T ||$ and makes the following diagram commutative, that is to say, $T = \hat{T} \circ \phi$.

$$\xymatrix{\mathbb{L}\ar_{\phi}[d]\ar[rr]^T&&X\\
F\ar_{\hat{T}}[urr]&& }$$

\end{defn}

Here, the norm of $T$ is $\norm{T} := \sup \set{\norm{T(x)}_X : x \in \mathbb{L}}$, while the norm of $\hat{T}$ is the usual for Banach spaces.

This definition determines a Banach lattice that we denote by $FBL\langle \mathbb{L}\rangle$ in an essentially unique way. When $\mathbb{L}$ is a distributive lattice (which is a natural assumption in this context, see Section~\ref{sectiondistributive}) the function $\phi$ is injective and, loosely speaking, we can view $FBL\langle \mathbb{L} \rangle$ as a Banach lattice which contains a subset lattice-isomorphic to $\mathbb{L}$ in a way that its elements work as free generators modulo the lattice relations on $\mathbb{L}$.

One of the main results in \cite{ART18} is an explicit description of $FBL[E]$ as a space of functions. The main result of this paper is also a description of $FBL\langle\mathbb{L}\rangle$ similar to that $FBL[E]$. In order to state this, define $$\mathbb{L}^{\ast} = \set {x^{\ast}: \mathbb{L} \longrightarrow [-1,1] : x^{\ast} \text{ is a lattice-homomorphism}}.$$ For every $x \in \mathbb{L}$ consider the evaluation map $\dot{\delta}_x : \mathbb{L}^{\ast} \longrightarrow \mathbb{R}$ given by $\dot{\delta}_x(x^{\ast}) = x^{\ast}(x)$. And for $f \in \mathbb{R}^{\mathbb{L}^{\ast}}$, define $$ \norm{f}_\ast = \sup \set{\sum_{i = 1}^n \abs{ f(x_{i}^{\ast})} : n \in \mathbb{N}, \text{ } x_1^{\ast}, \ldots, x_n^{\ast} \in \mathbb{L}^{\ast}, \text{ }\sup_{x \in \mathbb{L}} \sum_{i=1}^n \abs{x_i^{\ast}(x)} \leq 1 }.$$

\begin{thm}\label{main}
	Consider $F$ to be the Banach lattice generated by $\set{\dot{\delta}_x : x \in \mathbb{L}}$ inside the Banach lattice of all functions $f \in \mathbb{R}^{\mathbb{L}^{\ast}}$ with $\|f\|_\ast<\infty$, endowed with the norm $\|\cdot\|_\ast$ and the pointwise operations.
	Then $F$, together with the assignment $\phi(x)=\dot{\delta}_x$ is the free Banach lattice generated by $\mathbb{L}$.
\end{thm}

In spite of the similarity to the Banach space case from \cite{ART18}, our proof requires completely different techniques. Section~\ref{descriptionsection} is entirely devoted to this. In Section \ref{linearsection} we focus on the case when $\mathbb{L}$ is linearly ordered. Our main result in that section is that, for $\mathbb{L}$ linearly ordered, $FBL\langle\mathbb{L}\rangle$ has the countable chain condition (ccc) if and only if $\mathbb{L}$ is order-isomorphic to a subset of the real line. Remember that a Banach lattice has the countable chain condition if in every uncountable family of positive elements there are two whose infimum is not zero. This in contrast with the recent result that $FBL[E]$ has the ccc for every Banach space $E$ \cite{APRA}. In Section~\ref{sectionsumming} we check, also when $\mathbb{L}$ is linearly ordered, that the elements of $\mathbb{L}$ inside $FBL\langle\mathbb{L}\rangle$ behave like the summing basis of $c_0$ from a linear point of view.

\section{The Banach lattice $FBL\langle\mathbb{L}\rangle$ as a quotient of a space of functions}\label{existencesection}

Throughout this section $\mathbb{L}$ is a fixed lattice. Let us start by checking that Definition~\ref{FBLGBL} provides a uniquely determined object. If $\phi:\mathbb{L}\To F$ and $\phi':\mathbb{L}\To F'$ satisfy this definition, then we can get a Banach lattice homomorphism $\hat{\phi}':F\To F'$ with $\phi' = \hat{\phi}' \circ \phi$. Reversing the roles, we also get  $\hat{\phi}:F'\To F$ with $\phi = \hat{\phi} \circ \phi'$. The function $\hat{\phi}\circ\hat{\phi}'$ and the identity function $id_F$ on $F$ both satisfy Definition \ref{FBLGBL} as $\hat{T}$  when $T=\phi$. So
$\hat{\phi}\circ\hat{\phi}' = id_F$. Similarly, reversing roles, $\hat{\phi}'\circ\hat{\phi} = id_{F'}$. Thus, we obtained inverse lattice homomorphism of norm 1 between $F$ and $F'$ that commute with $\phi$ and $\phi'$.

Now, we are going to construct a Banach lattice $F$ that satisfies Definition~\ref{FBLGBL}. We will show later that the Banach lattice described in Theorem~\ref{main} also satisfies Definition~\ref{FBLGBL}. We take as a starting point that, when we view $\mathbb{L}$ as a set with no extra structure, we have the free Banach lattice $FBL(\mathbb{L})$, together with $u:\mathbb{L}\To FBL(\mathbb{L})$, constructed by de Pagter and Wickstead, whose universal property was described in the introduction. Take $\mathcal{I}$ the closed ideal of $FBL(\mathbb{L})$ generated by  
$$\set{u(x) \vee u(y) - u(x \vee y),\ \ u(x) \wedge u(y) - u(x \wedge y)\ : \ x,y \in \mathbb{L}}.$$ 
We take $F = FBL(\mathbb{L})/\mathcal{I}$, and $\phi: \mathbb{L} \longrightarrow FBL(\mathbb{L})/\mathcal{I}$ given by $\phi(x) = u(x) + \mathcal{I}$. The very definition of $\mathcal{I}$ provides that $\phi$ is a lattice homomorphism. Now, let $X$ be a Banach lattice and $T: \mathbb{L} \longrightarrow X$ a bounded lattice homomorphism. We know that $FBL(\mathbb{L})$ satisfies the universal property of free Banach lattices. Therefore, there exists a Banach lattice homomorphism $\hat{T}^1 : FBL(\mathbb{L}) \longrightarrow X$ such that $\hat{T}^1 \circ u = T$ and $\|\hat{T}^1\| = \norm{T}$. The fact that $T$ was a lattice homomorphism implies that $\hat{T}^1$ vanishes on $\mathcal{I}$. Thus, we can have a Banach lattice homomorphism $\hat{T}: FBL(\mathbb{L})/\mathcal{I} \longrightarrow X$ given by $\hat{T}(f + \mathcal{I}) = \hat{T}^1(f)$. It is clear that $\hat{T}\circ \phi = T$. Let us see that $\norm{T} = \| \hat{T} \|$. We only need to check that $\norm{T} \geq \| \hat{T} \|$. Let $f + \mathcal{I} \in FBL(\mathbb{L})/\mathcal{I}$ with $\norm{f}_{\mathcal{I}} < 1$. We have that $$\norm{f}_\mathcal{I} = \inf \set{\norm{f+g} : g \in \mathcal{I}},$$ and, therefore, there exists $g \in \mathcal{I}$ such that $\norm{f+g} < 1$. Thus, $\| \hat{T}(f + \mathcal{I}) \| = \| \hat{T}^1(f+g) \| \leq \norm{T}$. Only the uniqueness of the extension $T$ remains to be checked. But this follows from the uniqueness of the extension to $FBL(\mathbb{L})$, because if $\hat{T}\circ \phi = T$, then $\hat{T}\circ \pi \circ u = T$, where $\pi:FBL(\mathbb{L})\To FBL(\mathbb{L})/\mathcal{I}$ is the quotient map.

We have proven that $F=FBL(\mathbb{L})/\mathcal{I}$ together with $\phi$ above, satisfy Definition~\ref{FBLGBL}. To make this representation more concrete, let us recall  the description of the free Banach lattice $FBL(A)$ generated by a set $A$, as given in \cite[Corollary 2.9]{ART18}. For every $x\in A$, consider the evaluation map $\delta_x : [-1,1]^A \longrightarrow [-1,1]$, and for every $f:[-1,1]^A\To \mathbb{R}$, define
$$ \norm{f} = \sup \set{\sum_{i = 1}^n \abs{ f(x_{i}^{\ast})} : n \in \mathbb{N}, \text{ } x_1^{\ast}, \ldots, x_n^{\ast} \in [-1,1]^A, \text{ }\sup_{x \in A} \sum_{i=1}^n \abs{x_i^{\ast}(x)} \leq 1 }.$$
It is easy to check that the set $H$ of all functions $f$ with $\|f\|<\infty$ is a Banach lattice, when endowed with this norm and with the pointwise operations. The free Banach lattice $FBL(A)$ can be taken to be the Banach lattice generated by the functions $\delta_x$ inside $H$. The function $u$ would be $u(x) = \delta_x$.

\section{Distributivity}\label{sectiondistributive}

A lattice $\mathbb{L}$ is said to be distributive if the two operations $\wedge$ and $\vee$ distribute each other. That is, $a\wedge (b\vee c) = (a\wedge b) \vee (a\wedge c)$ and $a\vee (b\wedge c) = (a\vee b)\wedge (a\vee c)$ for all $a,b,c\in\mathbb{L}$. For a lattice $\mathbb{L}$, let $\widetilde{\mathbb{L}} = \phi(\mathbb{L})$ be the image of $\mathbb{L}$ inside $FBL\langle \mathbb{L}\rangle$. The following proposition collects some well known facts and observations:

\begin{prop}
For a lattice $\mathbb{L}$ the following are equivalent:
\begin{enumerate}
\item $\mathbb{L}$ is distributive,
\item $\mathbb{L}$ is lattice-isomorphic to a subset of a Boolean algebra,
\item $\mathbb{L}$ is lattice-isomorphic to a bounded subset of a Banach lattice,
\item The canonical map $\phi:\mathbb{L}\To FBL\langle\mathbb{L}\rangle$ is injective.
\end{enumerate}
\end{prop}

\begin{proof}
The equivalence of (1), (2) and (3) is well known, see \cite[Theorem II.19]{Gratzer} for $1\Rightarrow 2$, \cite[Theorem 1.b.3]{LinTza} for $2\Rightarrow 3$ and \cite[Proposition II.1.5]{Schaefer} for $3\Rightarrow 1$. It is obvious that (4) implies (3). If (3) holds, then we have a bounded injective lattice homomorphism $T:\mathbb{L}\To X$ for some Banach lattice $X$. Using Definition~\ref{FBLGBL}, there is $\hat{T}$ such that $\hat{T}\circ \phi = T$. Since $T$ is injective, $\phi$ is injective and therefore (4) holds.
\end{proof}

\begin{prop}\label{distrifree}
$FBL\langle\mathbb{L}\rangle = FBL\langle\widetilde{\mathbb{L}}\rangle$. More precisely, if $F$ with $\phi$ is the free Banach lattice over the lattice $\mathbb{L}$, then $F$ with the inclusion map is the free Banach lattice over the lattice $\widetilde{\mathbb{L}}$.
\end{prop}

The proof is immediate from Definition~\ref{FBLGBL}. The conclusion of these observations is that the most natural case in which to consider $FBL\langle\mathbb{L}\rangle$ is when $\mathbb{L}$ is distributive, and that the case of general $\mathbb{L}$ reduces to the distributive case in a natural easy way. Still, we find that it may be useful to state the results for any lattice $\mathbb{L}$. Two more facts: 

\begin{prop}\label{distrifactor}
Every lattice-homomorphism $x^\ast:\mathbb{L}\To [-1,1]$ factors through $\widetilde{\mathbb{L}}$. That is, there exists $y^\ast:\widetilde{\mathbb{L}}\To [-1,1]$ such that $x^\ast = y^\ast \circ \phi$.
\end{prop}

\begin{proof}
Find a Banach lattice homomorphism of norm at most one $\hat{x}^\ast:FBL\langle\mathbb{L}\rangle\To \mathbb{R}$ with $x^\ast = \hat{x}^\ast\circ\phi$, as in Definition~\ref{FBLGBL}. Take $y^\ast = \hat{x}^\ast|_{\widetilde{\mathbb{L}}}$.
\end{proof}

\begin{prop}
Every finitely generated sublattice of a distributive lattice is finite.
\end{prop}

\begin{proof}
This is a well known fact, see \cite[Lemma III.3]{Birkhoff}.
\end{proof}

\section{The Banach lattice $FBL\langle\mathbb{L}\rangle$ as a space of functions}\label{descriptionsection}

This section is devoted to the proof of Theorem~\ref{main}. Let $FBL_\ast\langle \mathbb{L} \rangle$ be the Banach lattice described in that theorem. By Propositions~\ref{distrifree} and~\ref{distrifactor}, both $FBL\langle\mathbb{L}\rangle$ and $FBL_\ast\langle\mathbb{L}\rangle$ remain unchanged if we change $\mathbb{L}$ by $\widetilde{\mathbb{L}}$. So we can assume along this section that $\mathbb{L}$ is distributive. Since we already know that $FBL(\mathbb{L})/\mathcal{I}$ is the free Banach lattice over the lattice $\mathbb{L}$, what we have to do is to find a Banach lattice isometry  $S:FBL(\mathbb{L})/\mathcal{I}  \To FBL_\ast\langle\mathbb{L}\rangle$ such that $S(\delta_x + \mathcal{I}) = \dot{\delta}_x$.

We know that $FBL(\mathbb{L}) = \overline{lat}^{\norm{\cdot}}\set{\delta_x : x \in \mathbb{L}} \subset \mathbb{R}^{[-1,1]^\mathbb{L}}$, where $$ \norm{f} = \sup \set{\sum_{i = 1}^n \abs{f(x_{i}^{\ast})} : n \in \mathbb{N}, \text{ } x_1^{\ast}, \ldots, x_n^{\ast} \in [-1,1]^\mathbb{L}, \text{ }\sup_{x \in \mathbb{L}} \sum_{i=1}^n \abs{x_i^{\ast}(x)} \leq 1 },$$ and recall that $FBL_\ast\langle \mathbb{L} \rangle = \overline{lat}^{\norm{\cdot}_\ast}\set{\dot{\delta}_x : x \in \mathbb{L}} \subset \mathbb{R}^{\mathbb{L}^\ast}$, where $$ \norm{f}_\ast = \sup \set{\sum_{i = 1}^n \abs{f(x_{i}^{\ast})} : n \in \mathbb{N}, \text{ } x_1^{\ast}, \ldots, x_n^{\ast} \in \mathbb{L}^{\ast}, \text{ }\sup_{x \in \mathbb{L}} \sum_{i=1}^n \abs{x_i^{\ast}(x)} \leq 1 }.$$ 

 For every function $f:[-1,1]^{\mathbb{L}}\To \mathbb{R}$, consider its restriction $R(f) = f \vert_{\mathbb{L}^\ast}$. It is clear that the function $R$ commutes with linear combination and the lattice operations and that $\norm{R(f)}_\ast \leq \norm{f}$. Moreover, $R(\delta_x) = \dot{\delta}_x$ for every $x\in\mathbb{L}$. From this, we conclude that if $f \in FBL(\mathbb{L})$, then $R(f) \in FBL_\ast\langle \mathbb{L} \rangle$, and we can view $R: FBL(\mathbb{L}) \longrightarrow FBL_\ast \langle \mathbb{L} \rangle$ as a Banach lattice homomorphism of norm 1. Moreover, since $\mathbb{L}^\ast$ consists of lattice homomorphisms, $R$ vanishes on the ideal $\mathcal{I}$ that we defined at the beginning of this section. Thus, we have a Banach lattice homomorphism of norm at most one $$R_{\mathcal{I}}: FBL(\mathbb{L})/\mathcal{I} \longrightarrow FBL_\ast\langle \mathbb{L} \rangle$$ given by $R_\mathcal{I}(f + \mathcal{I}) = R(f)$ for every $f + \mathcal{I} \in FBL(\mathbb{L})/\mathcal{I}$. What we want to prove is that $R_{\mathcal{I}}$ is an isometry. That is, we have to show that $$\norm{f}_{\mathcal{I}} \leq \norm{f\vert_{\mathbb{L}^\ast}}_\ast$$ for every $f \in FBL(\mathbb{L})$, where
$\|f\|_\mathcal{I} = \inf\{\|f+g\| : g\in \mathcal{I}\}$.

First, suppose that $\mathbb{L} = \set{0, \ldots, n-1} = n$ is finite. De Pagter and Wickstead showed that in this case, $FBL(\mathbb{L})$ consists exactly of all the positively homogeneous continuous functions on $[-1,1]^\mathbb{L} = [-1,1]^n$. Here, positively homogeneous means that $f(rx) = rf(x)$ whenever $r$ is a positive scalar. Moreover, if we consider the boundary $\partial [-1,1]^n$, and the Banach lattice of continuous functions $C(\partial [-1,1]^n)$, the restriction map $P:FBL(\mathbb{L})\To C(\partial [-1,1]^n)$ is a Banach lattice isomorphism (it is not however, an isometry: the norm of $FBL(\mathbb{L})$ is transferred to a lattice norm that is equivalent to the supremum norm).

A closed ideal in a lattice of continuous functions on a compact space always consists of the functions that vanish on a certain closed set. Thus, there exists a closed set $S\subset \partial [-1,1]^n$ such that
 $$\mathcal{I} = \set{f\in FBL(\mathbb{L}) : f\vert_S = 0}.$$

In fact, the points of $S$ must be those where $f$ vanish for all $f\in\mathcal{I}$, or equivalently, for all generators $f$ of $\mathcal{I}$: $$S = \set{(\xi_x)_{x \in \mathbb{L}}\in  \partial [-1,1]^n : \xi_x \vee \xi_y =\xi_{x \vee y},\, \xi_x \wedge \xi_y =\xi_{x \wedge y},\,x,y \in \mathbb{L}} = \mathbb{L}^\ast\cap \partial [-1,1]^n.$$

Now fix $f \in FBL(\mathbb{L})$, and let us prove that $\norm{f}_{\mathcal{I}} \leq \norm{f\vert_{\mathbb{L}^\ast}}_\ast$. Remember that $$ \norm{f\vert_{\mathbb{L}^\ast}}_\ast = \sup \set{\sum_{i = 1}^m \abs{f(x_{i}^{\ast})} : m \in \mathbb{N}, \text{ } x_1^{\ast}, \ldots, x_m^{\ast} \in \mathbb{L}^{\ast}, \text{ }\sup_{x \in \mathbb{L}} \sum_{i=1}^m \abs{x_i^{\ast}(x)} \leq 1 },$$ and  $$\norm{f}_{\mathcal{I}} = \inf \set{\norm{g} : g \in FBL(\mathbb{L}), f \sim_{\mathcal{I}} g}.$$ 

Given $k \in \mathbb{N}$, let $$S_k^+ = \set{x^\ast \in \partial [-1,1]^n : d(x^\ast,S)<\frac{1}{k}}$$ and $$S_k^- = \set{x^\ast \in \partial [-1,1]^n : d(x^\ast,S) \geq \frac{1}{k}}.$$

Since $S$ and $S_k^-$ are disjoint closed subsets of $\partial [-1,1]^n$, by Urysohn's lemma we can find a continuous function $\widetilde{1}_k : \partial [-1,1]^n\To [0,1]$ such that $\widetilde{1_k}(S) = 1$ and $\widetilde{1_k}(S_k^-) = 0$.

Define $f_k = P^{-1}(\widetilde{1_k}f|_S)\in FBL(\mathbb{L})$ be the positively homogeneous extension of $\widetilde{1_k}f|_S$ to the cube $[-1,1]^n$. Then $f_k \in FBL(\mathbb{L})$, and moreover,
 since $f_k|_S = f|_S$, we have that $f_k \sim_\mathcal{I} f$ for every $k$. Therefore, it is enough to prove that for a given $\varepsilon > 0$, there exists $k \in \mathbb{N}$ such that $\norm{f_k} \leq \norm{f\vert_{\mathbb{L}^\ast}}_\ast + \varepsilon$. 

We have that $$ \norm{f\vert_{\mathbb{L}^\ast}}_\ast = \sup \set{\sum_{i = 1}^m \abs{r_i f(x_{i}^{\ast})} : x_1^{\ast}, \ldots, x_m^{\ast} \in S, \text{ } r_1, \ldots, r_m \in \mathbb{R},\text{ }\sup_{x \in \mathbb{L}} \sum_{i=1}^m \abs{r_i x_i^{\ast}(x)} \leq 1 },$$  $$ \norm{f_k} = \sup \set{\sum_{i = 1}^m \abs{r_i f_k(x_{i}^{\ast})} : x_1^{\ast}, \ldots, x_m^{\ast} \in \partial [-1,1]^n, \text{ } r_1, \ldots, r_m \in \mathbb{R},\text{ }\sup_{x \in \mathbb{L}} \sum_{i=1}^m \abs{r_i x_i^{\ast}(x)} \leq 1 }.$$

Notice that the scalars $r_1, \ldots, r_m \in \mathbb{R}$ that appear in these formulas always satisfy $\sum_{i=1}^m|r_i|\leq n$. This is because for every $i$ we can find $\xi_i\in\mathbb{L}$ with $x_i^\ast (\xi)= \pm 1$, and then, $$ \sum_{i=1}^m \abs{r_i} = \sum_{\xi\in\mathbb{L}} \sum_{\xi_i = \xi} \abs{r_i x_i^\ast(\xi)} \leq \sum_{\xi\in\mathbb{L}} 1 = n.$$

The function $f$ is bounded and uniformly continuous on $[-1,1]^n$, so we can pick $k \in \mathbb{N}$ satisfying the following two conditions:
  $$ (1) \ \text{For all }x^\ast, y^\ast \in [-1,1]^n, \ \text{ if } d(x^\ast,y^\ast) \leq \frac{1}{k}, \text{ then }\abs{f(x^\ast) - f(y^\ast)} < \varepsilon/2n.$$
$$ (2) \ \frac{M n^2 }{n+k} < \frac{\varepsilon}{2}, \text{ where } M = \max\{|f(y^\ast)| : y^\ast \in [-1,1]^n\}.$$
 By the definition of $S_k^+$, given $x_i^\ast \in S_k^+$, there exists $y_i^\ast \in S$ such that $d(x_i^\ast, y_i^\ast) \leq \frac{1}{k}$. When $x_i^\ast\in S$, we can take $y_i^\ast = x_i^\ast$. In this way, we can estimate any sum in the supremum that gives $\norm{f_k}$ as follows:
\begin{eqnarray*} 
 \sum_{i=1}^m \abs{r_i f_k(x_i^\ast)} & = &  \sum_{x_i^\ast \in S_k^+} \abs{r_i f_k(x_i^\ast)} + \sum_{x_i^\ast \in S_k^-} \abs{r_i f_k(x_i^\ast)} \\ &=&   \sum_{x_i^\ast \in S_k^+} \abs{r_i f_k(x_i^\ast)} 
  \leq   \sum_{x_i^\ast \in S_k^+} \abs{r_i f(x_i^\ast)} \\
 & \leq &  \sum_{x_i^\ast \in S_k^+} \abs{r_i f(y_i^\ast)} + \sum_{x_i^\ast \in S_k^+} \abs{r_i}\abs{f(x_i^\ast) - f(y_i^\ast)} \\
 & \leq &  \sum_{x_i^\ast \in S_k^+} \abs{r_i f(y_i^\ast)} + \frac{\varepsilon}{2n}\sum_{x_i^\ast \in S_k^+} \abs{r_i}\\
 & \leq &  \sum_{x_i^\ast \in S_k^+} \abs{r_i f(y_i^\ast)} + \frac{\varepsilon}{2}. 
\end{eqnarray*}

We have estimated a sum in the supremum that gives $\norm{f_k}$ by something that looks very much like a sum in the supremum that gives 
$\norm{f\vert_{\mathbb{L}^\ast}}_\ast$. Still, in order to have a sum in that supremum we would need that $\sup_{x\in\mathbb{L}}\sum |r_iy_i^\ast(x)|\leq 1$. This is not the case, but we will get it after a small perturbation. For $x \in \mathbb{L}$,

\begin{eqnarray*}
\sum_{x_i^\ast \in S_k^+}\abs{r_i y_i^\ast(x)} & \leq & \sum_{x_i^\ast \in S_k^+} \abs{r_i x_i^\ast(x)} + \sum_{x_i^\ast \in S_k^+} \abs{r_i}\abs{y_i^\ast(x) - x_i^\ast(x)}\\
 & \leq &  \sum_{x_i^\ast \in S_k^+} \abs{r_i x_i^\ast(x)} + \frac{1}{k}\sum_{x_i^\ast \in S_k^+} \abs{r_i}\\ & \leq & 1 + \frac{n}{k}.
\end{eqnarray*}

Thus, the scalars $\widetilde{r_i} = \frac{r_i}{1 + n/k}$ and the elements $y_i^\ast$, for every $i$ with $x_i^\ast \in S_k^+$, are as required in the supremum that gives $\norm{f \vert_{\mathbb{L}^\ast}}_\ast$. Coming back to our estimate of the sum in the sup of $\norm{f_k}$:

\begin{eqnarray*} 
 \sum_{i=1}^m \abs{r_i f_k(x_i^\ast)}  & \leq &  \sum_{x_i^\ast \in S_k^+} \abs{r_i f(y_i^\ast)} + \frac{\varepsilon}{2}\\ 
  & \leq & \sum_{x_i^\ast \in S_k^+} \abs{\widetilde{r_i} f(y_i^\ast)} + \sum_{x_i^\ast \in S_k^+} \abs{(r_i - \widetilde{r_i}) f(y_i^\ast)} + \frac{\varepsilon}{2} \\
 & \leq & \norm{f \vert_{\mathbb{L}^\ast}}_\ast + \left(1 - \frac{1}{1 + n/k}\right)\sum_{x_i^\ast \in S_k^+} \abs{r_i f(y_i^\ast)} + \frac{\varepsilon}{2} \\
  & = & \norm{f \vert_{\mathbb{L}^\ast}}_\ast +  \frac{n}{n+k}\sum_{x_i^\ast \in S_k^+} \abs{r_i f(y_i^\ast)} + \frac{\varepsilon}{2} \\
    & \leq & \norm{f \vert_{\mathbb{L}^\ast}}_\ast +  \frac{Mn }{n+k}\sum_{x_i^\ast \in S_k^+} \abs{r_i} + \frac{\varepsilon}{2} \\
& \leq & \norm{f \vert_{\mathbb{L}^\ast}}_\ast +  \frac{Mn^2 }{n+k} + \frac{\varepsilon}{2} \leq \norm{f \vert_{\mathbb{L}^\ast}}_\ast +  \varepsilon, \\
\end{eqnarray*} as we needed to prove. This finishes the proof of Theorem~\ref{main} in the case when $\mathbb{L}$ is finite. Before getting to the infinite case, we state a lemma.

\begin{lem}\label{torpedo}
	Let $\mathbb{L}$ be a distributive lattice and $\mathbb{F}_0\subset \mathbb{L}$ be a finite subset. Then, there exists a finite sublattice $\mathbb{F}_1\subset\mathbb{L}$ such that for every lattice $\mathbb{M}$ and every lattice homomorphism $y^\ast:\mathbb{F}_1\To \mathbb{M}$ there exists a lattice homomorphism $z^\ast:\mathbb{L}\To \mathbb{M}$ such that $z^\ast|_{\mathbb{F}_0} = y^\ast|_{\mathbb{F}_0}$.
\end{lem}

\begin{proof}
		We start with a claim: If $\mathbb{M}$ is a finite lattice and $x^\ast:\mathbb{F}_0\To\mathbb{M}$ is a function which is not the restricion of any lattice homomorphism $z^\ast:\mathbb{L}\To \mathbb{M}$, then there exists a finite sublattice $\mathbb{F}_1[x^\ast]\subset \mathbb{L}$ that contains $\mathbb{F}_0$ and such that $x^\ast$ is not the restriction of any lattice homomorphism $y^\ast:\mathbb{F}_1[x^\ast]\To \mathbb{M}$. 
	
	Proof of the claim: For every finite subset $\mathbb{F}\subset \mathbb{L}$ that contains $\mathbb{F}_0$, consider the set
	\begin{eqnarray*}
		K_\mathbb{F} = \{ z^\ast:\mathbb{L}\To \mathbb{M} &:& z^\ast|_{\mathbb{F}_0} = x^\ast,\\ 
		& & z^\ast(a\wedge b) = z^\ast(a)\wedge z^\ast(b), \text{ for all } a,b\in\mathbb{F}, \\ 
		& & z^\ast(a\vee b) = z^\ast(a)\vee z^\ast(b), \text{ for all } a,b\in\mathbb{F}.   \}
	\end{eqnarray*}
	Since every finitely generated sublattice of a distributive lattice is finite, the negation of the claim above implies that $K_\mathbb{F} \neq\emptyset$ whenever $\mathbb{F}$ is finite. It is easy to check that $K_\mathbb{F}$ is a closed subset of $\mathbb{M}^\mathbb{L}$ (with the product topology of the discrete topology on $\mathbb{M}$). We also have that $\bigcap K_{\mathbb{F}^i} \supset K_{\bigcup\mathbb{F}^i}$ for any $\mathbb{F}^1,\ldots,\mathbb{F}^k$. Thus, the sets of the form $K_{\mathbb{F}}$ form a family of closeds subsets of $\mathbb{M}^\mathbb{L}$ with the finite intersection property. By compactness, there exists $z^\ast:\mathbb{L}\To \mathbb{M}$ that belongs to all sets $K_\mathbb{F}$. But then, $z^\ast$ is a lattice homomorphism with $z^\ast|_{\mathbb{F}_0} = x^\ast$ in contradiction with the hypothesis of the claim.
	
	Once the claim is proved, we return to the proof of the Lemma. First, let us notice that we can suppose that $\mathbb{F}_0$ is a finite sublattice of $\mathbb{L}$ and that $\mathbb{M}$ is finite. The first assumption is because we can pass to the sublattice generated by $\mathbb{F}_0$, and remember that every finitely generated distributive lattice is finite. The second assumption is because we can consider the restriction of $y^\ast$ onto its range. Let us say that two surjective lattice homomorphisms $x_1^\ast:\mathbb{F}_0\To \mathbb{M}_1$ and $x_2^\ast: \mathbb{F}_0\To \mathbb{M}_2$ are equivalent if there exists a lattice isomorphism $\phi:\mathbb{M}_1\To \mathbb{M}_2$ such that $\phi\circ x_1^\ast = x_2^\ast$. Clearly, there are only finitely many equivalence classes of such surjective lattice homomorphisms, so let $\mathcal{C} = \{x_1^\ast,x_2^\ast,\ldots,x_p^\ast\}$ be a finite list that contains a representative of each equivalence class. Let $\mathcal{C}'$ be the smallest list made of all the $x_i^\ast\in \mathcal{C}$ that are not the restriction of any lattice homomorphism $z^\ast:\mathbb{L}\To \mathbb{M}_i$. We can construct then $\mathbb{F}_1$ to be the sublattice of $\mathbb{L}$ generated by $\mathbb{F}_0$ and by all the $\mathbb{F}_1[x_i^\ast]$ for $x_i^\ast\in\mathcal{C}'$.
\end{proof}

Now, we consider the case when $\mathbb{L}$ is infinite. Again, we fix $g \in FBL(\mathbb{L})$, and have to show that $\norm{g}_{\mathcal{I}} \leq \norm{g\vert_{\mathbb{L}^\ast}}_\ast$.

For this proof it will be convenient to explicitly indicate the domain of the evaluation functions, so we write $\delta^\mathbb{L}_x:[-1,1]^\mathbb{L}\To \mathbb{R}$ for the function $\delta^\mathbb{L}_x(x^\ast) = x^\ast(x)$. We can suppose that $g$ can be written as $g = P(\delta^\mathbb{L}_{x_1}, \ldots, \delta^\mathbb{L}_{x_n})$ for some $x_1, \ldots, x_n \in \mathbb{L}$, where $P$ is a formula that involves linear combinations and the lattice operations $\wedge$ and $\vee$. This is because this kind of functions are dense in $FBL(\mathbb{L})$, that was generated by the functions $\delta^\mathbb{L}_x$ as a Banach lattice. Let $\mathbb{F}_0=\{x_1,\ldots,x_n\}$ and let $\mathbb{F}_1$ be the finite sublattice of $\mathbb{L}$ provided by Lemma~\ref{torpedo}. For any set $\mathbb{A}$ such that $\mathbb{F}_0 \subset \mathbb{A} \subset \mathbb{L}$, we consider
$$g^\mathbb{A} = P(\delta^\mathbb{A}_{x_1}, \ldots, \delta^\mathbb{A}_{x_n}):[-1,1]^\mathbb{A}\To \mathbb{R}$$

Claim X: If $\mathbb{A} \subset \mathbb{B}$ and $x^\ast\in [-1,1]^\mathbb{B}$, then $g^\mathbb{B}(x^\ast)  = g^\mathbb{A}(x^\ast|_\mathbb{A})$.

Proof of the claim: This is easily checked by induction on the complexity of the expression $P$. If $P$ is just a variable $P(u_1,\ldots,u_n) = u_i$, then we have the fact that $\delta^\mathbb{B}_{x_i}(x^\ast) = x^\ast(x_i) = \delta^\mathbb{A}_{x_i}(x^\ast|_\mathbb{A})$. And it is trivial that if the claim is satisfied by $P$ and $Q$, it is also satisfied for $P\wedge Q$, $P\vee Q$ and any linear combination of $P$ and $Q$. This finishes the proof of the claim.

Let $\mathcal{I}_1$ be the ideal of $FBL(\mathbb{F}_1)$ generated by the elements of the form $\delta^{\mathbb{F}_1}_{x\vee y} - \delta^{\mathbb{F}_1}_x \vee \delta^{\mathbb{F}_1}_y$ and $\delta^{\mathbb{F}_1}_{x\wedge y} - \delta^{\mathbb{F}_1}_x \wedge \delta^{\mathbb{F}_1}_y$. By the finite case that we already proved, we have that
 $$\norm{g^{\mathbb{F}_1}}_{\mathcal{I}_1} \leq \left\|g^{\mathbb{F}_1}\vert_{\mathbb{F}_1^\ast}\right\|_\ast.$$

Thus, it is enough to prove that $\norm{g}_{\mathcal{I}}\leq \norm{g^{\mathbb{F}_1}}_{\mathcal{I}_1}$ and that $\left\|g^{\mathbb{F}_1}\vert_{\mathbb{F}_1^\ast}\right\|_\ast \leq \| g \vert_{\mathbb{L}^\ast}\|_\ast$.

Let us see first that $\left\|g^{\mathbb{F}_1}\vert_{\mathbb{F}_1^\ast}\right\|_\ast \leq \| g \vert_{\mathbb{L}^\ast}\|_\ast$. We have that
$$\left\|g^{\mathbb{F}_1}\vert_{\mathbb{F}_1^\ast}\right\|_\ast = \sup \set{\sum_{i=1}^m\abs{g^{\mathbb{F}_1}(y_i^\ast)} : m \in \mathbb{N},  \text{ } y_i^\ast \in \mathbb{F}_1^\ast,\ \sup_{x \in \mathbb{F}_1} \sum_{i = 1}^m \abs{y_i^\ast(x)} \leq 1},$$ 

$$\| g \vert_{\mathbb{L}^\ast}\|_\ast = \sup \set{\sum_{i=1}^m\abs{g(z_i^\ast)} : m \in \mathbb{N},  \text{ } z_i^\ast \in \mathbb{L}^\ast, \text{ } \sup_{x \in \mathbb{L}} \sum_{i = 1}^m \abs{z_i^\ast(x)} \leq 1}.$$

We take a sum $\sum_{i=1}^m\abs{g^{\mathbb{F}_1}(y_i^\ast)}$ and we will find a sum $\sum_{i=1}^m\abs{g(z_i^\ast)}$  like in the second  supremum with the same value. Consider $$\mathbb{M} = \{(y_1^\ast(x),\ldots,y_m^\ast(x)) : x\in \mathbb{F}_1 \} \subset [-1,1]^m.$$ Notice that, since each $y_i^\ast$ is a lattice homomorphism, the set $\mathbb{M}$ is a sublattice of $\mathbb{R}^m$ and we have a lattice homomorphism $y^\ast:\mathbb{F}_1\To \mathbb{M}$ given by $y^\ast(x) = (y_1^\ast(x),\ldots,y_m^\ast(x))$. Also, since we are assuming that the $y_i^\ast$ are as in the supremum above, we have that $\sum_{i=1}^m |\xi_i| \leq 1$ whenever $(\xi_1,\ldots,\xi_m)\in \mathbb{M}$. We are in a position to apply Lemma~\ref{torpedo}, and we find a lattice homomorphism $z^\ast:\mathbb{L}\To\mathbb{M}\subset [-1,1]^m$ such that $z^\ast|_{\mathbb{F}_0} = y^\ast|_{\mathbb{F}_0}$. Write $z^\ast(x) = (z_1^\ast(x),\ldots,z_m^\ast(x))$, so that we have $z_1^\ast,\ldots,z_m^\ast\in \mathbb{L}^\ast$. Since the range of $z^\ast$ is inside $\mathbb{M}$, we have that $\sum_{i = 1}^m \abs{z_i^\ast(x)} \leq 1$ for all $x\in\mathbb{L}$. Finally, using Claim X above 
$$
\sum_{i=1}^m\abs{g(z_i^\ast)} = \sum_{i=1}^m\abs{g^\mathbb{L}(z_i^\ast)} = 
\sum_{i=1}^m\abs{g^{\mathbb{F}_0}(z_i^\ast|_{\mathbb{F}_0})} = 
\sum_{i=1}^m\abs{g^{\mathbb{F}_0}(y_i^\ast|_{\mathbb{F}_0})} =
\sum_{i=1}^m\abs{g^{\mathbb{F}_1}(y_i^\ast)},
$$
as required.

Now, we prove the remaining inequality $\norm{g}_{\mathcal{I}} \leq \| g^{\mathbb{F}_1} \|_{\mathcal{I}_1}$. In this proof, it will be useful to use a subindex on norms to indicate in which free Banach lattice these norms are calculated. Remember that

\begin{eqnarray*}
\norm{g}_{\mathcal{I}} &=& \inf \set{ \norm{f}_{FBL(\mathbb{L})} : f \in FBL(\mathbb{L}), \text{ }f -g \in {\mathcal{I}} },\\
\| g^{\mathbb{F}_1} \|_{\mathcal{I}_1} &=& \inf \set{\norm{h}_{FBL(\mathbb{F}_1)} : h \in FBL(\mathbb{F}_1), \text{ }h - g^{\mathbb{F}_1}\in \mathcal{I}_1},
\end{eqnarray*} where 

$$ \norm{f}_{FBL(\mathbb{L})} = \sup \set{\sum_{i = 1}^m \abs{ f(z_{i}^{\ast})} : m \in \mathbb{N}, \text{ } z_i^{\ast} \in [-1,1]^{\mathbb{L}}, \text{ }\sup_{x \in \mathbb{L}} \sum_{i=1}^m \abs{z_i^{\ast}(x)} \leq 1 },$$
$$\norm{h}_{FBL(\mathbb{F}_1)} = \sup \set{\sum_{i = 1}^m \abs{ h(y_{i}^{\ast})} : m \in \mathbb{N}, \text{ } y_i^\ast \in [-1,1]^{\mathbb{F}_1}, \text{ }\sup_{x \in \mathbb{F}_1} \sum_{i=1}^m \abs{y_i^{\ast}(x)} \leq 1 }.$$

Thus, the question is if given $h \in FBL(\mathbb{F}_1)$ such that $h-g^{\mathbb{F}_1}\in\mathcal{I}_1$, there exists $f \in FBL(\mathbb{L})$ such that $f-g\in  {\mathcal{I}}$ and $\norm{f}_{FBL(\mathbb{L})} \leq \norm{h}_{FBL(\mathbb{F}_1)}$.

For every $h:[-1,1]^{\mathbb{F}_1}\To \mathbb{R}$, we consider $e(h):[-1,1]^{\mathbb{L}}\To \mathbb{R}$ given by $e(h)(z^\ast) = h(z^\ast|_{\mathbb{F}_1})$. It is clear that $e(\delta_x^{\mathbb{F}_1}) = \delta_x^{\mathbb{L}}$, and $e$ preserves linear combinations, the lattice operations and $\|e(h)\|_{FBL(\mathbb{L})} = \|h\|_{FBL(\mathbb{F}_1)}$. Thus, we can view $e$ as a Banach lattice homomorphism $e:FBL(\mathbb{F}_1)\To FBL(\mathbb{L})$ that preserves the norm.

Now, we see that $f = e(h)$ is what we are looking for. It only remains to check that $f - g \in {\mathcal{I}}$. We know that $h - g^{\mathbb{F}_1}\in {\mathcal{I}_1}$, which is the ideal generated by $$\set{\delta^{\mathbb{F}_1}_{x\vee y} - \delta^{\mathbb{F}_1}_x \vee \delta^{\mathbb{F}_1}_y,\ \delta^{\mathbb{F}_1}_{x\wedge y} - \delta^{\mathbb{F}_1}_x \wedge \delta^{\mathbb{F}_1}_y : x, y \in \mathbb{F}_1}.$$

Therefore, $e(h) - e(g^{\mathbb{F}_1})$ is in the ideal generated by 
$$ \set{e\left(\delta^{\mathbb{F}_1}_{x\vee y} - \delta^{\mathbb{F}_1}_x \vee \delta^{\mathbb{F}_1}_y\right),\ e\left(\delta^{\mathbb{F}_1}_{x\wedge y} - \delta^{\mathbb{F}_1}_x \wedge \delta^{\mathbb{F}_1}_y\right) : x, y \in \mathbb{F}_1}. $$ $$ = \set{\delta^{\mathbb{L}}_{x\vee y} - \delta^{\mathbb{L}}_x \vee \delta^{\mathbb{L}}_y,\ \delta^{\mathbb{L}}_{x\wedge y} - \delta^{\mathbb{L}}_x \wedge \delta^{\mathbb{L}}_y : x, y \in \mathbb{F}_1}.$$

Notice that $e(g^{\mathbb{F}_1}) = g$ by Claim X above. So we conclude that $e(h) - e(g^{\mathbb{F}_1}) = f-g \in \mathcal{I}$ as required.

\section{Chain conditions on the free Banach lattice of a linear order}\label{linearsection}

Throughout this section $\mathbb{L}$ is a linearly ordered set, which is a particular case of a lattice, and $FBL\langle\mathbb{L}\rangle = FBL_\ast\langle\mathbb{L}\rangle$ is the free Banach lattice generated by $\mathbb{L}$, in the concrete form described in Theorem~\ref{main}. From now on, for $x\in\mathbb{L}$, we will denote the evaluation maps as $\delta_x:\mathbb{L}^\ast\To \mathbb{R}$ instead of $\dot{\delta}_x$, as we do not need to distinguish it anymore from other evaluation maps. A Banach lattice $X$ satisfies the \textit{countable chain condition} (\textit{ccc}), if whenever $\set{f_i : i \in I}\subset X$ are positive elements and $f_i \wedge f_j = 0$ for all $i \neq j$, then we must have that $\abs{I}$ is countable. This section is devoted to the proof of the following result:

\begin{thm}\label{FBLinearccc}
For $\mathbb{L}$ linearly ordered, $FBL\langle\mathbb{L}\rangle$ has the countable chain condition if and only if $\mathbb{L}$ is order-isomorphic to a subset of the real line.
\end{thm}

We first state a couple of lemmas:

\begin{lem}\label{reallines}
For a linearly ordered set $\mathbb{L}$ the following are equivalent:
\begin{enumerate}
\item $\mathbb{L}$ is order-isomorphic to a subset of the real line.
\item $\mathbb{L}$ is separable in the order topology, and the set of leaps $\{(a,b) \in \mathbb{L}^2 : [a,b]=\{a,b\}\}$ is countable.
\item For every uncountable family of triples $$\mathcal{F} = \big\lbrace \lbrace x_1^{i}, x_2^i, x_3^{i} \rbrace : x_1^{i}, x_2^i, x_3^{i} \in \mathbb{L},\, x_1^{i}< x_2^i < x_3^{i},\, i \in J \big\rbrace$$ there exist $i \neq j$ such that $x_1^i\leq x_2^j\leq x_3^i$ and $x_1^j \leq x_2^i\leq x_3^j$.
\end{enumerate}
\end{lem}

\begin{proof}
The equivalence of (1) and (2) is easy and is well known folklore, cf. \cite[Corollary 3.1]{Todlines}. Assume now (2) and let us prove (3). Take a countable dense subset $D \subset \mathbb{L}$ that contains all the leaps $\set{(a,b) \in \mathbb{L}^2 : [a,b] = \{a,b\}}\subset D$. Let $f: \mathcal{F} \longrightarrow D^{2}$ be the map given by $f(x^i_1,x^i_2,x^i_3) = (d_1,d_2)$, where $d_k$ is an element of $D$ such that $x^i_k < d_k < x^i_{k+1}$ if such an element exists, and $d_k = x^i_k$ otherwise, when $x^i_k,x^i_{k+1}\in D$ form a leap. Since $\mathcal{F}$ is uncountable and ${D}^{2}$ is countable, there exists an uncountable $\mathcal{F}_0 \subseteq \mathcal{F}$ such that $f\vert_{\mathcal{F}_0}$ is constant. Any pair of distinct elements $\lbrace x_1^i, x_2^i, x_3^i \rbrace$, $\lbrace x_1^j, x_2^j, x_3^j \rbrace \in \mathcal{F}_0$ is as required because we can interpolate $x_k^u \leq d_k < x_{k+1}^v$. Let us prove now that (3) implies (2). First, let us see that the set of leaps is countable. Let us say two leaps $(a,b)$ and $(a',b')$ are equivalent if there exist $c_0<c_1<\cdots<c_p$ finitely many elements of $\mathbb{L}$ such that each $(c_k,c_{k+1})$ is a leap and either $c_0=a$ and $c_p=b'$, or $c_0=a'$ and $c_p=b$. It is clear that each equivalence class of leaps is countable. So if there were uncountably many leaps, we could find an uncountable family $\mathcal{G} = \{ \{x_1^i,x_2^i\} :  i\in J\}$ of nonequivalent leaps $x_1^i < x_2^i$. We can asume that $x_2^i$ is never the maximum of $\mathbb{L}$ and we choose an arbitrary $x_3^i>x_2^i$. Applying (3) to the family $\mathcal{F} = \{ \{x_1^i,x_2^i,x_3^i\} :  i\in J\}$ , we could find $i\leq j$ such that $x_1^{i} \leq x_2^j$ and $x_1^j \leq x_2^{i}$. But when we have two nonequivalent leaps, one has to be strictly to the right of the other, so either $x_2^j < x_1^i$ or $x_2^i < x_1^j$, a contradiction. Now we prove that $\mathbb{L}$ is separable. Using Zorn's lemma, we can find a maximal family $\mathcal{F}$ that fails the property stated in (3). This family must be then countable. Let $D$ be the set of all elements of $\mathbb{L}$ that either appear in some triple of the family $\mathcal{F}$ or are one of the two sides of a leap. We know now that $D$ is countable. Let us check that it is dense. Take a nonempty open interval $(a,b)\subset \mathbb{L}$. If the interval $(a,b)$ is finite, then all its elements are parts of leaps, so it intersects $D$. Suppose that $(a,b)$ is infinite but does not intersect $D$. Then if we pick $a<x_1<x_2<x_3<b$, then the triple $\{x_1,x_2,x_3\}$ could be added to $\mathcal{F}$, in contradiction with its maximality.
\end{proof}

We notice that the use of triples in Lemma~\ref{reallines} is essential. The analogous property of condition (3) for couples instead of triples would be that for every uncountable family $\mathcal{F} = \{\{x^i_1,x^i_2\} : x^i_1<x^i_2\}$ there are $i\neq j$ such that $x^i_1\leq x^j_2$ and $x^j_1\leq x^i_2$. A connected Suslin line has this weaker property but it does not embed inside the real line. 

\begin{lem}\label{sublinear}
Let $\mathbb{L}\subset\mathbb{M}$ be two linearly ordered sets. Then $FBL\langle\mathbb{L}\rangle$ is isometric to a closed sublattice of $FBL\langle\mathbb{M}\rangle$.
\end{lem}

\begin{proof}
We will prove that, in fact, $FBL\langle \mathbb{L} \rangle$ is isometric to the closure of the vector lattice generated by the image of $\mathbb{L}$ via the inclusion mapping of $\mathbb{M}$ inside $FBL \langle \mathbb{M} \rangle$. That is, if we denote by $\delta^{\mathbb{M}}: \mathbb{M} \longrightarrow FBL\langle \mathbb{M} \rangle$ the inclusion of $\mathbb{M}$ inside $FBL\langle \mathbb{M} \rangle$ (given by $\delta^{\mathbb{M}}(x)(x^*) = \delta_x^{\mathbb{M}}(x^*) = x^*(x)$, for every $x \in \mathbb{M}$, $x^* \in \mathbb{M}^*$), we have that $$FBL \langle \mathbb{L} \rangle \cong \overline{lat}^{\norm{\cdot}_\ast}\set{\delta_x^{\mathbb{M}} : x \in \mathbb{L}} \subset FBL\langle \mathbb{M} \rangle.$$
To prove that, let us also denote by $\delta^{\mathbb{L}}: \mathbb{L} \longrightarrow FBL\langle \mathbb{L} \rangle$ the inclusion of $\mathbb{L}$ inside $FBL\langle \mathbb{L} \rangle$ in a similar way to $\delta^{\mathbb{M}}$, and let $\varphi : \mathbb{R}^{\mathbb{L}^*} \longrightarrow \mathbb{R}^{\mathbb{M}^*}$ the map given by $\varphi (f)(x^*) = f(x^* \vert_{\mathbb{L}})$, for every $f \in \mathbb{R}^{\mathbb{L}^*}$, $x^* \in \mathbb{M}^*$.

It is clear that the function $\varphi$ commutes with linear combinations and the lattice operations, and that $\norm{\varphi(f)}_* \leq \norm{f}_*$ for every $f \in \mathbb{R}^{\mathbb{L}^*}$, where 

\begin{eqnarray*}
\norm{f}_* & = & \sup \set{\sum_{i = 1}^n \abs{f(x_{i}^{*})} : n \in \mathbb{N}, \text{ } x_1^{*}, \ldots, x_n^{*} \in \mathbb{L}^{*},\text{ }\sup_{x \in \mathbb{L}} \sum_{i=1}^n \abs{x_i^{*}(x)} \leq 1 },
\end{eqnarray*}

\begin{eqnarray*}
\norm{\phi(f)}_* & = & \sup \set{\sum_{i = 1}^n \abs{\phi(f)(y_{i}^{*})} : n \in \mathbb{N}, \text{ } y_1^{*}, \ldots, y_n^{*} \in \mathbb{M}^{*}, \text{ }\sup_{x \in \mathbb{M}} \sum_{i=1}^n \abs{y_i^{*}(x)} \leq 1 } \\
& = & \sup \set{\sum_{i = 1}^n \abs{f(y_{i}^{*}\vert_{\mathbb{L}})} : n \in \mathbb{N}, \text{ } y_1^{*}, \ldots, y_n^{*} \in \mathbb{M}^{*}, \text{ },\text{ }\sup_{x \in \mathbb{M}} \sum_{i=1}^n \abs{y_i^{*}(x)} \leq 1 }.
\end{eqnarray*}
Moreover, $\varphi(\delta_x^{\mathbb{L}}) = \delta_x^{\mathbb{M}}$ for every $x \in \mathbb{L}$. Thus, to see that $\varphi$ gives an isometry from $FBL\langle \mathbb{L} \rangle$ onto $\overline{lat}^{\norm{\cdot}_\ast}\set{\delta_x^{\mathbb{M}} : x \in \mathbb{L}} \subset FBL\langle \mathbb{M} \rangle$, it only remains to prove that we also have $\norm{f}_* \leq \norm{\varphi (f)}_*$ for every $f \in FBL\langle \mathbb{L} \rangle$. 

First, observe that $$\mathbb{L}^\ast = \set{x^\ast: \mathbb{L} \longrightarrow [-1,1] : u\leq v \Rightarrow x^\ast(u)\leq x^\ast(v) }$$ and $$\mathbb{M}^\ast = \set{x^\ast: \mathbb{M} \longrightarrow [-1,1] : u\leq v \Rightarrow x^\ast(u)\leq x^\ast(v) }.$$ 

Fix $f \in FBL\langle \mathbb{L} \rangle$ and let $x_1^*, \ldots, x_n^* \in \mathbb{L}^*$ like in the expression of the norm $\norm{f}_*$. Let $\gamma : \mathbb{L}^* \longrightarrow \mathbb{M}^*$ be the map given by 
\begin{equation*}
     \gamma(x^*)(x) = \left\{
	       \begin{array}{ll}
		 \sup \set{x^*(y) : y \in \mathbb{L}, y \leq x}      & \mathrm{if\ } \text{there exists }y \in \mathbb{L} \text{ with }y \leq x, \\
		\inf \set{x^*(y) : y \in \mathbb{L}, y \geq x}      & \text{otherwise},\\
	       \end{array}
	     \right.
\end{equation*}
for every $x^* \in \mathbb{L}^*$, $x \in \mathbb{M}$.

Put $y_i^* := \gamma(x_i^*) \in \mathbb{M}^*$ for every $i = 1, \ldots, n$, and let us see that $y_1^*, \ldots, y_n^*$ are like in the expression of the norm $\norm{\varphi(f)}_*$ satisfying that $\sum_{i=1}^n|f(x_i^*)| \leq \sum_{i=1}^n|f(y_i^*\vert_{\mathbb{L}})|.$ 

Since $\gamma(x^*) \vert_{\mathbb{L}} = x^*$ for every $x^* \in \mathbb{L}^*$, we have that $$f(y_i^* \vert_\mathbb{L}) = f(\gamma(x_i^*) \vert_\mathbb{L}) = f(x_i^*)$$ for every $i = 1, \ldots, n.$

Finally, we have to check that $\sup_{x \in \mathbb{M}} \sum_{i=1}^n |y_i^*(x)| \leq 1$.

Suppose not, and let $x \in M$ and $\varepsilon > 0$ such that $\sum_{i=1}^n |y_i^*(x)| > 1 + \varepsilon$. Suppose also that there exists $y \in \mathbb{L}$ with $y \leq x$ (the other case is analogous).

Since, in this case, $y_i^*(x) = \sup \set{x_i^*(y) : y \in \mathbb{L}, y \leq x}$ for every $i = 1, \ldots, n$, we have that there exists $\tilde{y_i} \in \mathbb{L}$, with $\tilde{y_i} \leq x$, such that $y_i^*(x) - x_i^*(\tilde{y_i}) < \frac{\varepsilon}{n}$. Now, if $y \in \mathbb{L}$ is such that $\tilde{y_i} \leq y \leq x$, since $x_i^*$ is increasing, we have that $x_i^*(\tilde{y_i}) \leq x_i^*(y)$. But then, $|x_i^*(y) - y_i^*(x)| = y_i^*(x) - x_i^*(y) \leq y_i^*(x) - x_i^*(\tilde{y_i}) < \frac{\varepsilon}{n}$.

Let $\tilde{y} := \max \set{\tilde{y_1}, \ldots, \tilde{y_n}} \in \mathbb{L}$. Due to the above, we have that $|x_i^*(\tilde{y}) - y_i^*(x)| < \frac{\varepsilon}{n}$ for every $i = 1, \ldots, n$. Then, using that $\sum_{i=1}^n|y_i^*(x)| > 1 + \varepsilon$ and $\sum_{i=1}^n |x_i^*(\tilde{y}) - y_i^*(x)| < \varepsilon$, we have that
\begin{eqnarray*} 
 \sum_{i=1}^n |x_i^*(\tilde{y})| & = &   \sum_{i=1}^n |x_i^*(\tilde{y}) - y_i^*(x) + y_i^*(x)| \\ 
 & \geq & \sum_{i=1}^n|y_i^*(x)| - \sum_{i=1}^n|x_i^*(\tilde{y}) - y_i^*(x)| \\
 & > & 1 + \varepsilon - \varepsilon = 1,
\end{eqnarray*}
which is a contradiction.
\end{proof}

We prove now Theorem \ref{FBLinearccc}. Endow $\mathbb{L}^\ast$ with the pointwise topology. If a function $f:\mathbb{L}^\ast\To \mathbb{R}$ belongs to $FBL\langle\mathbb{L}\rangle$, then it is continous. This is because the functions $\delta_x$ are continuous, and the property of being continous is preserved under all Banach lattice operations (including limits, because every limit in $FBL\langle\mathbb{L}\rangle$ is a uniform limit).

A basis for the topology of $\mathbb{L}^\ast$ is given by the sets of the form
$$U(x_1, I_1, \ldots, x_n, I_n) := \set{x^\ast \in \mathbb{L}^\ast : x^\ast(x_i) \in I_i \text{ for all } i=1,\ldots,n}.$$
for $x_1,\ldots,x_n\in\mathbb{L}$ and $I_1,\ldots,I_n$ open intervals with rational endpoints. Write $I_i<I_j$ if $\sup(I_i)<\inf(I_j)$, and consider the family $$\mathcal{W} = \set{ U(x_1, I_1, \ldots, x_n, I_n) : x_1<x_2<\cdots<x_n, I_1<I_2<\cdots<I_n }.$$
This is not a basis anymore. But since $\mathbb{L}^\ast$ consists of nondecreasing functions, it is clear that $\mathcal{W}$ is a $\pi$-basis. That means that every nonempty open subset of $\mathbb{L}^\ast$ contains a nonempty open subset from $\mathcal{W}$.

Let us suppose that $\mathbb{L}$ is a subset of the real line, and we prove that $FBL\langle\mathbb{L}\rangle$ is ccc. Let $D\subset\mathbb{L}$ be a countable dense subset of $\mathbb{L}$ that contains all element that are part of a leap, $D\supset\{a,b : [a,b] = \{a,b\} \}$. Observe that in this case
$$\mathcal{W}_0 = \set{ U(d_1, I_1, \ldots, d_n, I_n)\in\mathcal{W} : d_1,d_2,\cdots,d_n\in D}$$ 
is also a $\pi$-basis of $\mathbb{L}^\ast$. This is because for every $U(x_1, I_1, \ldots, x_n, I_n)\in\mathcal{W}$, we can interpolate $d_1^-\leq x_1\leq d_1^+\leq d_2^-\leq x_2\leq d_2^+ \leq \cdots \leq d_n^-\leq x_n \leq d^+_n$ with $d^\pm_k\in D$, and then
$$U(d^-_1, I_1,d^+_1,I_1, \ldots, d^-_n, I_n,d_n^+,I_n) \subset U(x_1, I_1, \ldots, x_n, I_n). $$
Take an uncountable family of positive elements $\mathcal{G}\subset FBL\langle \mathbb{L} \rangle$. For each $f\in \mathcal{G}$ there exists $V_f \in \mathcal{W}_0$ such that $V_f \subset  \{x^\ast\in\mathbb{L}^\ast : f(x^\ast)>0\}$. Notice that $f\wedge g \neq 0$ whenever $V_f\cap V_g\neq \emptyset$. Since $\mathcal{G}$ is uncountable and $\mathcal{W}_0$ is countable, there are plenty of pairs $f,g$ such that in fact $V_f = V_g$. This finishes the proof that $FBL\langle \mathbb{L} \rangle$ is ccc whenever $\mathbb{L}$ embeds in the real line.

We may notice that we proved a property stronger that the ccc: If a linear order $\mathbb{L}$ embeds into the real line, then $FBL\langle \mathbb{L} \rangle$ is $\sigma$-centered. That means, we can decompose the positive elements into countably many pieces in such a way that every finite infimum inside each piece is nonzero.

Now we turn to the proof that if $\mathbb{L}$ does not embed into the real line, then $FBL\langle \mathbb{L} \rangle$ is not ccc.  We are going to prove it first under the extra assumption that $\mathbb{L}$ has a maximum $M$ and a minimum $m$. We fix an uncountable family of triples $\mathcal{F}$ that fails property (3) in Lemma~\ref{reallines}. For every $i\in J$ consider
$$h_i = 0 \vee \left( \delta_{x_1^i}\wedge \left( \delta_{x_2^{i}} - \delta_{x_{1}^{i}} - 0.4\ \delta_M \right) \wedge \left( \delta_{x_3^{i}} - \delta_{x_{2}^{i}} - 0.4\ \delta_M \right)\right).$$ 
Let us see that these elements of $FBL\langle\mathbb{L}\rangle$ witness the failure of the ccc. Obviously $h_i\geq 0$. First, we fix $i$ and we check that  $h_i > 0$. For this, define $x^\ast : \mathbb{L} \longrightarrow [-1,1]$ by
\begin{equation*}
     x^\ast(x) = \left\{
	       \begin{array}{ll}
		 0.1      & \mathrm{if\ } x < x_2^i, \\
		 0.55 & \mathrm{if\ } x_2^i \leq x < x_{3}^i, \\
		 1 & \mathrm{if\ } x_3^i \leq x.\\
	       \end{array}
	     \right.
\end{equation*}
We have that $h_i(x^\ast) = 0 \vee\left(0.1\wedge \left( 0.55 -0.1 - 0.4 \right) \wedge \left( 1 - 0.55 - 0.4 \right)\right) = 0.05$, so $h_i\neq 0$.

Now, we prove that $h_i \wedge h_j = 0$ for $i \neq j$. Suppose on the contrary that $h_i \wedge h_j > 0$. Then, there exists $x^\ast \in \mathbb{L}^\ast$ such that $h_i(x^\ast) \wedge h_j(x^\ast) > 0$. Then  $$x^\ast(x_1^i)>0,\ x^\ast(x^j_1)>0,$$
$$x^\ast(x_2^i) - x^\ast(x_{1}^i) > 0.4\ x^\ast(M),$$
$$x^\ast(x_3^i) - x^\ast(x_{2}^i) > 0.4\ x^\ast(M),$$ $$x^\ast(x_2^j) - x^\ast(x_{1}^j) > 0.4\ x^\ast(M),$$ 
$$x^\ast(x_3^j) - x^\ast(x_{2}^j) > 0.4\ x^\ast(M).$$
Remember that property (3) of Lemma~\ref{reallines} fails, and therefore either $x_2^j\not\in [x_1^i,x_3^i]$ or $x_2^i\not\in [x_1^j,x_3^j]$. For example, say that $x_2^i < x_1^j$ (all other cases are analogous). Then, combining the fact that $x^\ast$ is nondecreasing with the above inequalities, we get that
\begin{eqnarray*}
x^\ast(M) &>& x^\ast(M)-x^\ast(x_1^i)\\ &\geq& x^\ast(x_3^j) - x^\ast(x_1^i)  \\
 & = & x^\ast(x_3^j) - x^\ast(x_2^j)
+ x^\ast(x_2^j) - x^\ast(x_1^j) + \\ & &  x^\ast(x_1^j) - x^\ast(x_2^i)
+ x^\ast(x_2^i) - x^\ast(x_1^i)\\
& > & 1.2 \ x^\ast(M),
\end{eqnarray*}
a contradicition because $x^\ast(M)\geq x^\ast(x^i_1)>0$.

The proof of the case when $\mathbb{L}$ has a maximum is over. Let $\overleftarrow{\mathbb{L}}$ be the linear order whose underlying set is the same as $\mathbb{L}$, but with the reverse order. It is easy to check that the map $\Phi : FBL\langle\mathbb{L}\rangle\To FBL\langle \overleftarrow{\mathbb{L}} \rangle$ given by $\Phi(f)(x^\ast) = -f(-x^\ast)$ is an isomorphism of Banach lattices with $\Phi(\delta_x) = \delta_x$ for all $x\in\mathbb{L}$. Thus, $FBL\langle\mathbb{L}\rangle$ and $FBL\langle \overleftarrow{\mathbb{L}} \rangle$ are isomorphic, so we will have that $\mathbb{L}$ embeds into the real line whenever $FBL\langle\mathbb{L}\rangle$ is ccc and $\mathbb{L}$ has a minimum. The case when $\mathbb{L}$ has neither a maximum nor a minimum remains. In that case, we just pick an arbitrary element $a\in\mathbb{L}$ and consider
$\mathbb{L}_1 = \{x\in\mathbb{L} : x\leq a\}$ and $\mathbb{L}_2 = \{x\in\mathbb{L} : x\geq a\}$. By Lemma~\ref{sublinear}, if $FBL\langle\mathbb{L}\rangle$ is ccc then both $FBL\langle\mathbb{L}_1\rangle$ and $FBL\langle\mathbb{L}_2\rangle$ are ccc. But $\mathbb{L}_1$ and $\mathbb{L}_2$ have a maximum and a minimum respectively, so by the cases that we already proved, we conclude that both $\mathbb{L}_1$ and $\mathbb{L}_2$ embed into the real line. This implies that $\mathbb{L}$ embeds into the real line, as required.

\section{Linear structure of a line in its free Banach lattice}\label{sectionsumming}

In this section, $\mathbb{L}$ is again a linearly ordered set, and $FBL\langle\mathbb{L}\rangle$ its free Banach lattice, in the form of Theorem~\ref{main}, with embedding $\phi:\mathbb{L}\To FBL\langle\mathbb{L}\rangle$ given by $\phi(x)=\delta_x$. We will show that in this case, the linear combinations of the copy of $\mathbb{L}$ inside $FBL\langle\mathbb{L}\rangle$ behave similarly to the summing basis of $c_0$. More precisely:

\begin{prop}
	Let $\mathbb{L}$ be a linearly ordered set. Then, for every $u_1 < \ldots < u_m \in \mathbb{L}$ and $a_1, \ldots, a_m \in \mathbb{R}$ we have that
	$$\left\| \sum_{i=1}^{m}a_i s_i\right\|_{\infty} \leq \left\| \sum_{i=1}^{m}a_i\delta_{u_i}\right\|_* \leq 6 \left\|\sum_{i=1}^{m}a_i s_i\right\|_{\infty},$$ where $s_i = (\underbrace{1,1,\ldots,1}_i,0,0,0,\ldots) \in c_0$.
	
\end{prop}

\begin{proof}
	Let $T: \mathbb{L} \longrightarrow c_0$ be the map given by
	\begin{equation*}
		T(x) = \left\{
		\begin{array}{ll}
			s_1      & \mathrm{if\ } x < u_2; \\
			s_k & \mathrm{if\ } u_k \leq x < u_{k+1} \text{ for any }k \geq 2. \\
		\end{array}
		\right.
	\end{equation*}
	Clearly, $T$ is a bounded and increasing map. Let $\hat{T}: FBL \langle \mathbb{L} \rangle \longrightarrow c_0$ be its extension as in Definition~\ref{FBLGBL}. Since $\| \hat{T} \| \leq 1$, we have that $\| \hat{T}(\sum_{i=1}^m a_i\delta_{u_i})\|_{\infty} \leq \| \sum_{i=1}^m a_i \delta_{u_i} \|_*$, where $\hat{T}(\sum_{i=1}^m a_i\delta_{u_i}) = \sum_{i=1}^m a_is_i$. This proves the first inequality in the proposition.

	For $f \in FBL_*\langle \mathbb{L} \rangle$ we have that 	
	\begin{eqnarray*} 
		\norm{f}_* & = & \sup \set{\sum_{j = 1}^n \abs{ f(x_{j}^{*})} : n \in \mathbb{N}, \text{ } x_1^{*}, \ldots, x_n^{*} \in \mathbb{L}^{*}, \text{ }\sup_{x \in \mathbb{L}} \sum_{j=1}^n \abs{x_j^{*}(x)} \leq 1 } \\
		&  \leq & 2\sup \set{\left| \sum_{j = 1}^n  f(x_{j}^{*})\right| : n \in \mathbb{N}, \text{ } x_1^{*}, \ldots, x_n^{*} \in \mathbb{L}^{*}, \text{ }\sup_{x \in \mathbb{L}} \sum_{j=1}^n \abs{x_j^{*}(x)} \leq 1 }.
	\end{eqnarray*}
	This is because $$\sum_{j=1}^n |f(x_j^*)| = \left|\sum_{f(x_j^*) > 0}f(x_j^*)\right| + \left|\sum_{f(x_j^*) < 0}f(x_j^*)\right|.$$
	Therefore 
		\begin{eqnarray*} 
		\left\| \sum_{i=1}^m a_i\delta_{u_i}\right\|_*
		&  \leq & 2\sup \set{ \left| \sum_{j = 1}^n \sum_{i=1}^m a_i x_j^*(u_i) \right| : n \in \mathbb{N}, \text{ } x_1^{*}, \ldots, x_n^{*} \in \mathbb{L}^{*}, \text{ }\sup_{x \in \mathbb{L}} \sum_{j=1}^n \abs{x_j^{*}(x)} \leq 1 } \\
		& = & 2\sup \set{ \left| \sum_{i = 1}^m a_i (\sum_{j=1}^n x_j^*)(u_i) \right| : n \in \mathbb{N}, \text{ } x_1^{*}, \ldots, x_n^{*} \in \mathbb{L}^{*}, \text{ }\sup_{x \in \mathbb{L}} \sum_{j=1}^n \abs{x_j^{*}(x)} \leq 1 }\\
		& \leq & 2\sup \set{\left| \sum_{i=1}^m a_i x^*(u_i)\right| : x^*  \in \mathbb{L}^*}.
	\end{eqnarray*}
	
	On the other hand,
	$$ 3\left\| \sum_{i=1}^m a_is_i \right\|_{\infty}  =  \sup \set{\left| z^*\left(\sum_{i=1}^m a_i s_i\right) \right| : z^* \in 3B_{{c_0}^*}} = \sup \set{\left| \sum_{i=1}^m a_iz^*(s_i) \right| : z^* \in 3B_{\ell_1}}.$$
	
	Given $x^* \in \mathbb{L}^*$, if we define $z_1 = x^*(u_1)$ and $z_k = x^*(u_k) - x^*(u_{k-1})$ for every $k \geq 2$, then $z^* = (z_1,z_2,z_3,\ldots) \in 3B_{\ell_1}$, and $z^\ast(s_i) = x^\ast(u_i)$ for all $i=1,\ldots,m$. Combining all these facts, we get the second inequality in the proposition.	
\end{proof}

\end{document}